\documentclass[11pt,a4paper,reqno,undertilde]{article}
\usepackage{amscd,amsfonts,amssymb,amsthm,amsmath, comment}
 \usepackage{bm} 
\usepackage{setspace} 
\usepackage{tikz}
\usepackage{amsrefs}
\usetikzlibrary{cd, arrows} 
\usepackage[latin1]{inputenc}

\def\Ga{{\Gamma}}
\def\vp{\varphi}
\def\bbn{\mathbb{N}}

\theoremstyle{plain}
  \newtheorem{prop}{Proposition}[section]
  \newtheorem{thm}[prop]{Theorem}
  \newtheorem{cor}[prop]{Corollary}

  \newtheorem{rem}[prop]{Remark}
    \newtheorem{lem}[prop]{Lemma}
  \newtheorem*{lem*}{Lemma}
  \newtheorem*{claim*}{Claim}
\newtheorem{defn}[prop]{Definition}
  \newtheorem*{defn*}{Definition}
  
  \newtheorem*{rem*}{Remarks}
\def\bbr{\mathbb{R}}
\def\bbc{\mathbb{C}}
\def\g{\mathfrak{g}}

\newcommand{\GL}{\operatorname{GL}}
\newcommand{\SL}{\operatorname{SL}}
\newcommand{\SO}{\operatorname{SO}}
\newcommand{\SU}{\operatorname{SU}}
\newcommand{\Sp}{\operatorname{Sp}}
\newcommand{\prof}{\operatorname{PR}}
\newcommand{\id}{\operatorname{id}}
\newcommand{\Hom}{\operatorname{Hom}}
\newcommand{\Ext}{\operatorname{Ext}}
\newcommand{\tg}{\operatorname{tg}}
\newcommand{\gr}{\operatorname{gr}}
\newcommand{\const}{\operatorname{const}}
\newcommand{\rank}{\operatorname{rk}}

\newcommand{\bbN}{{\mathbb N}}
\newcommand{\bbQ}{{\mathbb Q}}
\newcommand{\bbR}{{\mathbb R}}

\newcommand{\bbZ}{{\mathbb Z}}
\newcommand{\bbC}{{\mathbb C}}
\newcommand{\bbH}{\mathbb{H}}

\newcommand{\bfG}{{\mathbf G}}

\newcommand{\calD}{\mathcal{D}}

\newcommand{\calO}{\mathcal{O}}

\begin{document}

\setcounter{section}{0}

\title{Stability and Instability\\ of Lattices in Semisimple Groups}
\author{Uri Bader, Alexander Lubotzky, Roman Sauer and Shmuel Weinberger}
\date{}

\maketitle


\medskip

\begin{abstract}
Using cohomological methods, we show that lattices in semisimple groups are typically stable with respect to the Frobenius norm but not with respect to the operator norm. 
\end{abstract}

\section{Introduction} 

Let $\Ga$ be a group and $\g = (G_n, d_n)^\infty_{n=1} $ a family of groups $G_n$ equipped with bi-invariant metrics $d_n$.  An \emph{ asymptotic homomorphism} $\vp = \{ \vp_n\}^\infty_{n=1}$ from $\Ga $ to $\g$  is a family of maps $\vp_n:\Ga \to G_n$ such that for every $x,y \in \Ga$,
\[\lim\limits_{n\to\infty} \; d_n(\vp_n(x) \vp_n(y), \; \vp_n(xy)) = 0.\]
The group $\Gamma$ is $\g$-\emph{approximated} if there is a separating asymptotic homomorphism to $\g$, i.e. there is $\vp = \{ \vp_n\}^\infty_{n=1}$ such that $\varlimsup d_n(\varphi_n(x),1_{G_n})>0$ for every $1\neq x\in\Gamma$.
The group $\Ga $ is said to be $\g${\emph{-stable}} if for every asymptotic homomorphism $\vp = \{ \vp_n \}^\infty_{n=1}$ there exists a ``true" homomorphism $\psi = \{ \psi_n\}^\infty_{n=1}$, s.t. for every $n \in \bbn$, $\psi_n:\Ga \to G_n$ is an homomorphism  and for every $x \in \Ga, \lim\limits_{n\to\infty} d_n(\psi_n(x), \vp_n(x)) = 0$.  Namely, every asymptotic homomorphism is a small deformation of true homomorphism.
It is a standard fact that in case $\Gamma=\langle S\mid R\rangle$ is a finitely presented group defined by a finite set of generators $S$ and a finite set of relations $R$, $\Gamma$ is $\g$-stable iff for every $\epsilon>0$ there exists $\delta>0$ such that for every $n$ and for every map $\rho:S\to G_n$, if for every word $w=s_1\cdots s_m$ in $R$, $d_n(\rho(s_1)\cdots \rho(s_m),1_{G_n})<\delta$ then there exists a group homomorphism $\bar{\rho}:\Gamma\to G_n$ such that for every $s$ $S$, $d_n(\rho(s),\bar{\rho}(s))<\epsilon$, c.f 
\cite[Definition 3.2]{Arzhantseva}.

The question of stability has been an intensive direction of research in recent years, (see \cite{frob} and the references therein).
One of the motivations to study stability is that it provides a path to prove $\frak{g}$-{\bf in}approximability which has been (and still is) a major open problem for various $\frak{g}$'s, but was solved for some $\frak{g}$'s using stability, see \cite{frob} and \cite{LO}.

The current paper concerns the case where $\Ga$ is an irreducible lattice in a semisimple Lie group. By a semisimple group $G$ we mean a product of non-compact (real) simple groups $G = \mathop{\prod}\limits^l_{i = 1} G_i$. Denote $r_i = \rank_\bbr (G_i)$ and $\rank(G) = r = \sum\limits^l_{i = 1} r_i$.  We say that $G$ is \emph{higher rank} if $r \ge 2$.
Irreducible lattices in higher-rank semisimple groups form an interesting class of groups enjoying the property of super-rigidity, for example, and they are all arithmetic.

In this paper we will concern ourselves with the family of unitary groups $\g = \{ U_n(\bbc)\}^\infty_{n=1}$ with two different metrics, both derived from two different norms on $M_n(\bbc)$

\begin{description}

  \item[(F)] The Frobenius norm (the $L^2$-norm), i.e., for $A \in M_n (\bbc), \; \| A \|_F = tr (A^*A)$ and $d^F_n (g, h) = \| g - h\|_F$ for $g, h \in U_n (\bbc)$

  \item[(O)] \ The operator norm (the $L^\infty$-norm) i.e. for $A \in M_n (\bbc), \; \| A \|_{op} = \\ \mathop{\sup}\limits_{\substack{v \in\bbc^n\\ \| v \| = 1}}$
    $ \|Av\|$ and $d^{op}_n (g, h) = \| g - h\|_{op}$.

\end{description}

Our results say that higher-rank lattices behave very differently with respect to these two metrics.  We show that ``almost all" (and conjecturally all) such lattices are Frobenius stable but at the same time are operator instable.

To this end, let us recall the cohomological criterions for Frobenius stability and operator instability given in \cite{frob} and \cite{Dadarlat}, respectively.

\begin{thm}\label{thm: main old result} Let $\Ga$ be a finitely generated group.
  \begin{description}
    \item[(F)] \cite{frob} If $H^2(\Ga, V) = 0$ for every unitary representation of $\Ga$ on any Hilbert space $V$, then $\Ga$ is Frobenius stable
    \item[(O)] \cite{Dadarlat} Assume $\Ga$ is a linear group and $H^i(\Ga,\bbr) \neq 0$ for some {\bf even} $i > 0$.  Then $\Ga$ is {\bf not} operator stable.
  \end{description}
\end{thm}

Thm 1.1(F) \cite{frob} gave several examples of Frobenius stable groups, namely lattices in simple $p$-adic Lie groups of rank $\geq 3$.
This relies on the work of Garland \cite{Garland} and his followers (\cite{frob} and the reference therein), who proved cohomological vanishing below the rank for such $p$-adic groups, that is property $[T_{\rank-1}]$ as defined below. However, it left open the case of real simple groups. This is particularly frustrating, as Garland's method was designed as a $p$-adic analogue of the Matsushima formula that was first proven in the context of real Lie groups. 
In fact, the situation for real Lie groups is not as neat and clean as the in the p-adic case and the cohomology below the rank often  be non-zero. 

The main goal is to settle the annoying discrepancy between the p-adic and real worlds.
What enables us doing so is the recent work of the first and the third author~\cite{Bader-Sauer} which shows a vanishing result below the rank.
However, in the real case the vanishing does not hold for the trivial representation, that is the groups satisfy property $(T_{\rank-1})$, but in general not $[T_{\rank-1}]$, as defined below.

\begin{defn}
  A group $\Gamma$ has property $[T_n]$ if its cohomology $H^i(\Gamma, V)$ vanishes for every unitary $\Gamma$-representation $V$ and every $1\le i\le n$. It has property $(T_n)$ if this holds provided $V$ has no non-zero $\Gamma$-invariant vectors.
\end{defn}

It is a remarkable but easy fact that for $n=1$, property $(T_1)$ is equivalent to property $[T_1]$, and both properties are equivalent to Kazhdan's property $(T)$ for countable groups~(see~\cite{Bader-Sauer}*{Lemma~3.2}).
However, for every $n>1$ there exist examples of groups satisfying property $(T_n)$ but not property $[T_n]$ (e.g $\text{Sp}_{2n+2}(\mathbb{Z})$).
Theorem~\ref{thm: main old result}, in particular, says that a $[T_2]$-group is Frobenius stable. The lattices in Theorem~\ref{thm: main stability result} below are only $(T_2)$-groups, not necessarily $[T_2]$-groups. In spite of that, they are Frobenius stable.

As mentioned above, the starting point of this paper is the recent paper~\cite{Bader-Sauer} which shows that an irreducible lattice $\Gamma$ in a higher rank semisimple Lie group $G$ has property $(T_{n-1})$ where $n$ is the minimal rank of each non-compact factor of~$G$~\cite{Bader-Sauer}*{Theorem~E}. Theorem~\ref{thm: main old result} strongly suggests that such $\Gamma$ is Frobenius stable if the rank of each non-compact factor is at least~$3$ but falls short from proving it as there might be a non-trivial second cohomology for the trivial representation. For example, the group $G=\Sp(2g,\bbR)$ has a universal central extension $\tilde G$ with an infinite abelian kernel which implies that $H_c^2(G,\bbR)$ and $H^2(\Gamma,\bbR)$ are non-zero. So $\Gamma$ does not have property $[T_2]$ despite having property $(T_2)$.

To handle this difficulty we argue in two steps. First we show that for $\Gamma$ which has property $(T_2)$, its preimage $\tilde\Gamma$ in $\tilde G$ is a $[T_2]$-groups. To this end, we study the exterior square $\Gamma\wedge\Gamma$ of $\Gamma$.  We show that $\Gamma\wedge\Gamma$ is a $[T_2]$-group provided $\Gamma$ is a $(T_2)$-group (Theorem~\ref{thm: univext}). If $\Gamma$ is perfect then $\Gamma\wedge\Gamma$ is the universal central extension of~$\Gamma$. Then we use another result of \cite{Bader-Sauer} to deduce that $\Gamma\wedge\Gamma$ and $\tilde\Gamma$ are commensurable and hence $\tilde\Gamma$ is also a $[T_2]$-group. Theorem~\ref{thm: main old result} now shows that $\tilde\Gamma$ is Frobenius stable. To deduce that $\Gamma$ is also Frobenius stable we use a generalization of a result of Deligne~\cite{De}. See Theorem~\ref{thm: deligne} which shows that $\tilde\Gamma$ modulo its profinite radical is commensurable with $\Gamma$ provided $\Gamma$ has the congruence subgroup property. We then conclude the following theorem.

\begin{thm}\label{thm: main stability result}
  Let $G$ be a real semisimple Lie group and $\Gamma<G$ be an  irreducible lattice. We assume that each non-compact factor of $G$ has rank at least~$3$. Then $\Gamma$ is Frobenius stable with the possible exception that one of the non-compact factors is of Hermitian type and $\Gamma$ does not have the congruence subgroup property.
\end{thm}

According to a well known conjecture of Serre~\cite{ser72} all the lattices in Theorem~\ref{thm: main stability result} satisfy the congruence subgroup property. This has been proved in most cases, in particular $G=\Sp(2g,\bbR)$, see \cite[Theorem 3]{Rapinchuk}. So we believe that there are no exceptions in the theorem. Moreover, we believe that all irreducible lattices in higher rank semisimple Lie groups are Frobenius stable.

The theorem above and the conjectural generalization are especially interesting in view of~\cite{BL} which shows, in particular, that these lattices are never Hilbert-Schmidt stable. Recall that the Hilbert-Schmidt norm is nothing more than the renormalization of Frobenius norm, i.e. $\|A\|_\mathrm{HS}=n^{-1/2}\|A\|_\mathrm{F}$ for $A\in M_n(\bbC)$.

Before passing to the operator norm case, let us mention an interesting corollary of Theorem~\ref{thm: main stability result}. In \cite{frob} some groups which cannot be Frobenius approximated were presented for the first time. These are certain central extensions of lattices in $p$-adic Lie groups which are somewhat difficult to describe, and the proof appealed to Garland's theory. Here are some examples which are easier to describe.

\begin{cor}\label{cor: non approximation result}
  Let $G=\Sp(2g,\bbR)$, $g\ge 3$, and let $\tilde G$ its universal central extension. For every lattice $\Gamma<G$, in particular for $\Gamma=\Sp(2g,\bbZ)$, the preimage $\tilde \Gamma$ of $\Gamma$ in $\tilde G$ is not Frobenius approximated.
\end{cor}

It is an interesting fact that Deligne's work~\cite{De} is used in two different ways to prove the above theorem and its corollary. First we use the characterization of the profinite radical to conclude the stability of $\Gamma$ from the one of $\tilde\Gamma$. Second we use the fact that there is a profinite radical, so $\tilde\Gamma$ is not residually finite, to conclude that $\tilde\Gamma$ cannot be Frobenius approximated.

The result described so far will be described in full and proved in \S2. In \S3 we will switch our attention to the case of operator norm.
Here one gets a fairly complete result for uniform lattices.

\begin{thm} \label{thm:mainO}
    Let $G$ be a semisimple real Lie group and $\Gamma<G$ a {\bf cocompact} lattice. 
    Then $\Gamma$ is a {\bf not} operator stable, with the only possible exceptions being when $G$ is locally isomorphic to either $\SO(n,1)$ for $n$ odd or $\SL_3(\mathbb{R})$.
\end{thm}

Theorem~\ref{thm:mainO} will follow from Theorem~\ref{thm: main old result} by showing that the uniform lattices in question carry some non-trivial even degree cohomology.
We now sketch the proof of the above statement.
We note that Matsushima's formula~\cite{Matsushima} (or the later Shapiro Lemma) yields an injection $H^*_c(G,\mathbb{R})\hookrightarrow H^*(\Gamma,\mathbb{R})$,
where $H^*_c$ denotes the continuous group cohomology and $\Gamma<G$ is uniform.
Thus it suffices to show that $G$ carries some non-trivial even degree continuous group cohomology.
We will use the fundamental fact that the continuous group cohomology of a semisimple Lie group is isomorphic to the \emph{singular} cohomology of the corresponding dual symmetric space, which is a compact manifold~\cite{Borel-Wallach}*{Theorem~IX.5.6}. Observe that a compact manifold carries a non-trivial even degree cohomology iff it is not an odd dimensional rational homology sphere. 
Indeed, if the dimension is even then the top class is such a non-trivial cohomology class, while if the dimension is odd, the non-triviality of any odd degree cohomology implies the non-triviality of its even complimentary degree cohomology.
We are thus left to show that under the conditions of Theorem~\ref{thm:mainO}, the dual symmetric space of $G$ is not an odd dimensional rational homology sphere. 
Fortunately, the compact symmetric spaces which are rational homology spheres were completely classified by Wolf~\cite[Theorem 1]{Wolf}. In particular, the only odd dimensional rational homology spheres are actual odd dimensional spheres (or the corresponding projective spaces), with a unique exception that corresponds to the dual symmetric space of $\SL_3(\mathbb{R})$. 
The final step in the proof of Theorem~\ref{thm:mainO} follows by applying a classification by Borel of all the presentations of odd dimensional spheres as homogeneous spaces, which in particular shows that these form Riemannian symmetric spaces only in the cases $\SO(n+1)/\SO(n)$ for odd $n$~\cite{Borel},
and noting that the latter are the dual symmetric spaces of $\SO(n,1)$ for odd $n$.

In \S3.1 we will consider lattices in those semisimple Lie groups not covered by Theorem~\ref{thm:mainO}, namely the groups $\SL_3(\mathbb{R})$ and $\SO(n,1)$ for $n$ odd. 
We will show in Corollary~\ref{cor:SL3} that lattices in $\SL_3(\mathbb{R})$ have finite index subgroups which are not operator stable and the situation with lattices in $\SO(n,1)$ will be summarized in Corollary~\ref{cor:SOn1}.

Finally, in \S4 we will suggest and call the attention to various open problems and further research directions.

\subsection{Acknowledgements}

The authors wish to thank the referees who helped improving the paper in many places.

U.B was partially supported by ISF Moked grant number 2919/19.
A.L acknowledges with gratitude the hospitality and support of the Fields Institute
(Toronto)  where part of this work was carried on, as well as a grant by the European Research Council (ERC)
under the European Union's Horizon 2020 (grant agreement No 882751). 
R.S was partially supported by the projects 441426599 and 441848266 funded by the DFG (Deutsche Forschungsgemeinschaft). R.S. thanks the University of Chicago and the Weizmann Institute for the hospitality when part of this work was carried out. 
S.W was partially supported by the NSF grant 2105451.

\section{Frobenius stability}

\subsection{The transgression associated with a central extension}

We consider a group $\Gamma$, an abelian group $A$ and a central extension 
 \begin{equation} \label{eq:pi}
 1\to A\to L \overset{\rho}{\longrightarrow} \Gamma\to 1.
   \end{equation}
  Let $k$ be another abelian group. 
  For an element $\beta\in \Hom(A,k)$ we consider the graph of $-\beta$, 
  \[ \gr(-\beta)=\{(a,-\beta(a))\mid a\in A\}< A \times k, \]
  as a central subgroup of $L\times k$ and consider the associated quotient group 
  \[ L^\beta=L \times k/\gr(-\beta). \]
  Loosely speaking, $L^\beta$ is a version of $L$ in which $A$ is replaced by $k$ via $\beta$.
  The composition of $\rho$ with the projection $L\times k\to L$ induces a homomorphism $\tg_\rho(\beta)\colon L^\beta \to \Gamma$. 
  We obtain a commutative diagram whose rows are central extensions. 
  \begin{equation} \label{eq:beta}
 \begin{tikzcd}
      1 \arrow[r] & A \arrow[r] \arrow[d, "\beta"] & L \arrow[d]\arrow[r, "\rho"] & \Gamma \arrow[d]\arrow[r] & 1 \\
            1 \arrow[r] & k \arrow[r] & L^\beta \arrow[r, "\tg_\rho(\beta)"] & \Gamma \arrow[r] & 1 
\end{tikzcd}
 \end{equation}
  where the maps $L\to L^\beta$ and $k\to L^\beta$ are induced by 
  $\id\times 0:L\to L\times k$ and $0\times \id:k\to L\times k$ correspondingly.
We consider this extension as an element of $H^2(\Gamma,k)$, thus we obtain a map 

\begin{equation} \label{eq:tgpi}
\tg_\rho: H^1(A,k)\cong \Hom(A,k) \to H^2(\Gamma,k), \quad \beta\mapsto \tg_\rho(\beta). 
\end{equation}
By the explicit description given in \cite[Proposition (1.6.6)]{neukirch}, this map is the \emph{transgression} with $k$ coefficients associated with the extension \eqref{eq:pi} and 
by \cite[Proposition (1.6.7)]{neukirch}, it fits in the associated Inflation-Restriction five terms exact sequence  
  \begin{align} \label{eq:5t} 
   0 \to H^1(\Gamma,k) \to H^1(L,k) \to 
   H^1(A,k) \overset{\tg_\rho}{\longrightarrow} H^2(\Gamma,k)  \to  H^2(L,k). 
  \end{align}

The following lemma will be needed in the proof of Theorem~\ref{thm: univext}.

\begin{lem} \label{lem:tg}
Assume $k$ is a a field of characteristic 0 and the transgression $\tg_\rho$ is an isomorphism.
Then $H^1(L,k) \cong  H^1(\Gamma,k)$ and $H^2(L,k)$ embeds into  $H^1(\Gamma,k) \otimes H^1(A,k)$.
In particular, if $\Gamma$ has a finite abelianization then $H^2(L,k)=H^1(L,k)=0$.
\end{lem}

\begin{proof}
The fact that $H^1(\Gamma,k) \cong H^1(L,k)$ follows immediately from the exactness of \eqref{eq:5t}
using the assumption that $\tg_\rho$ is an isomorphism.
We proceed to show that $H^2(L,k)$ embeds into $H^1(\Gamma,k) \otimes H^1(A,k)$.

We consider the Hochschild-Serre-Lyndon spectral sequence associated  with the central extension~\eqref{eq:pi}. Its $E_2$-term is given by 
  \begin{equation}\label{eq: iso of algebras} E_2^{pq}=H^p(\Gamma, H^q(A,k))\cong H^p(\Gamma,k)\otimes  H^q(A, k).
  \end{equation}
  We are interested in the line $p+q=2$.
The term $E_\infty^{11}=E_3^{11}$ embeds into $E_2^{11}\cong H^1(\Gamma,k) \otimes H^1(A,k)$. 
Hence it is enough to show that $E_\infty^{02}=E_\infty^{20}=0$.
By~\cite[Proposition (2.4.3)]{neukirch}, $d_2^{01}=\tg_\rho$.
Since this is an isomorphism, in particular a surjection, we deduce that $E_\infty^{20}=E_3^{20}=0$.
We are left to show that $E_\infty^{02}=E_4^{02}=0$.
We will do so by showing that $d^{02}_2\colon E_2^{02}\to E_2^{21}$ is injective, thus $E_3^{02}=0$.
Roughly speaking, $d^{02}_2$ is injective because it could be identified with the symmetrization map of an exterior algebra. The proof below will make this statement precise.

  Note that the isomorphism \eqref{eq: iso of algebras} represents the $E_2$-term as graded algebra, namely the product of the graded algebras $H^\ast(\Gamma,k)$ and $H^\ast(A, k)$.
  Upon identifying the $H^0$-factors in the $E_2$-term with~$k$ we may view $d_2^{01}$ as a map $\tilde d_2^{01}\colon H^1(A,k)\to H^2(\Gamma,k)$. Note that $d_2^{01}(1\otimes x)=\tilde d_2^{01} (x)\otimes 1$ under this notation and recall that this is an isomorphism. Furthermore,
  we have
  \begin{equation}\label{eq: algebra abelian group} H^\ast(A,k)\cong \Lambda^\ast H^1(A,k).
  \end{equation}
  Consider the following diagram
  \[\begin{tikzcd}
      \Lambda^2 H^1(A,k)\arrow[d, "\eqref{eq: algebra abelian group}"', "\cong"]\ar[r, "\sigma"] & H^1(A,k) \otimes H^1(A,k)\ar[d, "\tilde d_2^{01}\otimes\id"', "\cong"]\\
      H^0({\Gamma},k)\otimes H^2(A,k)\arrow[d,"\eqref{eq: iso of algebras}"',"\cong"] & H^2({\Gamma},k) \otimes H^1(A,k)\arrow[d, "\eqref{eq: iso of algebras}"',"\cong"]\\
      E_2^{02}\arrow[r, "d_2^{02}"] & E_2^{21}
    \end{tikzcd}
  \]
  where $\sigma$ is the symmetrization map $x\wedge y\mapsto x\otimes y-y\otimes x$. 
  
  Fix $x\wedge y\in \Lambda^2 H^1(A,k)$. It corresponds to
  $(1\otimes x)\cdot (1\otimes y)\in H^0(\Gamma,k)\otimes H^2(A,k)$. By the derivation property,
  \begin{align*} d_2^{02}\bigl( (1\otimes x)\cdot (1\otimes y)\bigr) & = d_2^{01}(1\otimes x)\cdot (1\otimes y)-(1\otimes x)\cdot d_2^{01}(1\otimes y)                                         \\
                                                                   & = \bigl(\tilde d_2^{01}(x)\otimes 1\bigr) \cdot (1\otimes y)-(1\otimes x)\cdot \bigl( \tilde d_2^{01}(y)\otimes 1\bigr) \\
                                                                   & = \tilde d_2^{01}(x)\otimes y-\tilde d_2^{01}(y)\otimes x.
  \end{align*}
  This shows that the diagram commutes.
  As the symmetrization map $\sigma$ over a field of characteristic 0 is injective, we conclude that $d^{02}_2$ is injective. 
\end{proof}

\subsection{The exterior square}

In this subsection we review an important construction that we believe should be better known, the \emph{exterior square} of a group and the associated short exact sequence \eqref{eq:extsq}.
The exterior square was first defined and studied in \cite{Miller}
but it was named only much later, in \cite[\S2.7]{Brown-Loday}, in the context of study of non-abelian tensor products.
A thorough and clear presentation of the subject is given in the thesis \cite{Naik}.

Let $\Gamma$ be a group.
The \emph{(non-abelian) exterior square} of $\Gamma$, denoted
 $\Gamma\wedge \Gamma$, is the group generated by the symbols $x\wedge y$, for every $x,y\in \Gamma$, under the relations
\[ xx'\wedge y=(xx'x^{-1}\wedge xyx^{-1})(x\wedge y), \quad x\wedge yy'=(x\wedge y)(yxy^{-1}\wedge yy'y^{-1}) \]
and 
\[ x\wedge x=1 \]
for every $x,x',y,y'\in \Gamma$
(the first line of relations defines the tensor square, $\Gamma\otimes \Gamma$, and the second line provides a natural homomorphism $\Gamma\otimes \Gamma\to \Gamma\wedge \Gamma$). The relations are such that $x\wedge y\mapsto [x,y]$ induces a  homomorphism 
$\Gamma\wedge \Gamma\to [\Gamma,\Gamma]$. 
  We now consider an abelian group $k$ and a central extension 
   \begin{equation} \label{eq:pik}
 1\to k\to L \overset{\pi}{\longrightarrow} \Gamma\to 1.
   \end{equation}
   We get a natural homomorphism 
   \begin{equation} \label{eq:barpi}
   \bar{\pi}:\Gamma \wedge \Gamma\to [L,L], \quad x\wedge y \mapsto [s(x),s(y)],
   \end{equation}
where $s:\Gamma\to L$ is a set theoretical section for $\pi$.
Indeed, one checks easily that $[s(x),s(y)]$ does not depend on the choice of the section $s$,
thus induces a well defined map $\Gamma\times \Gamma \to [L,L]$ which extends to $\Gamma \wedge \Gamma$ using the standard commutator relations.
In particular, the identity central extension $\id:\Gamma\to\Gamma$ gives rise to a group homomorphism $\bar{\id}:\Gamma \wedge \Gamma\to [\Gamma,\Gamma]$
and we get $\bar{\id}=\pi|_{[L,L]}\circ \bar{\pi}$.
By the main result of \cite{Miller}
we have $\ker(\bar{\id})\cong H_2(\Gamma,\mathbb{Z})$, which gives us the central extension 
 \begin{equation} \label{eq:extsq}
      1\to H_2(\Gamma,\mathbb{Z})\to \Gamma\wedge \Gamma\overset{\bar{\id}}{\longrightarrow} [\Gamma, \Gamma]\to 1. 
  \end{equation}
  See~\cite[\S2.7]{Brown-Loday} and~\cite[\S3.8.4 and \S3.4.1]{Naik}. 
  We obtain an element $h(\pi)\in \Hom(H_2(\Gamma,\bbZ),k)$ fitting the following commutative diagram.
\begin{equation} \label{eq:h}
\begin{tikzcd}
      1 \arrow[r] & H_2(\Gamma,\mathbb{Z}) \arrow[r] \arrow[d, "h(\pi)"] & \Gamma\wedge \Gamma \arrow[d, "\bar{\pi}"]\arrow[r] & {[\Gamma,\Gamma]} \arrow[d]\arrow[r] & 1 \\
            1 \arrow[r] & k \arrow[r] & L \arrow[r, "\pi"] & \Gamma \arrow[r] & 1 
\end{tikzcd}
  \end{equation}
  Interpreting the extension $\pi$ as an element in $H^2(\Gamma,k)$ we obtain a map 
  \[ h:H^2(\Gamma,k)\to \Hom(H_2(\Gamma,\bbZ),k), \quad \pi \mapsto h(\pi). \]
  By~\cite[\S3.6.4]{Naik}, the map $h$ is the epimorphism appearing in the Universal Coefficient Theorem: 
\begin{equation} \label{eq:uct}
0 \to \Ext^1_\bbZ(\Gamma/[\Gamma,\Gamma],k) \to H^2(\Gamma,k) \overset{h}{\longrightarrow} \Hom(H_2(\Gamma,\bbZ),k) \to 0 
\end{equation}

Next, we specialize \eqref{eq:uct} to $k=H_2(\Gamma,\bbZ)$ and, using the fact that $h$ is surjective we find a central extension $\pi_0$ in $H^2(\Gamma,H_2(\Gamma,\bbZ))$ such that $h(\pi_0)$ is the identity in $\Hom(H_2(\Gamma,\bbZ),H_2(\Gamma,\bbZ))$.
The commutative diagram \eqref{eq:h} specializes to the following commutative diagram
\begin{equation}\label{eq:E_0}
\begin{tikzcd} 
      1 \arrow[r] & H_2(\Gamma,\mathbb{Z}) \arrow[r] \arrow[d, "h(\pi_0)=\id"] & \Gamma\wedge \Gamma \arrow[d, "\bar{\pi}_0"]\arrow[r] & {[\Gamma,\Gamma]} \arrow[d]\arrow[r] & 1 \\
            1 \arrow[r] & H_2(\Gamma,\mathbb{Z}) \arrow[r] & L_0 \arrow[r, "\pi_0"] & \Gamma \arrow[r] & 1 
\end{tikzcd}
\end{equation}
A central extension, such as $\pi_0$, fitting in~\eqref{eq:E_0} is called a \emph{Schur covering} of $\Gamma$. We note that Schur coverings form initial objects in the category of extensions of $\Gamma$ where morphism are taken to be \emph{homoclinisms}~\cite[\S3.6.5]{Naik}.

We now consider the transgression~\eqref{eq:tgpi} 
\[\tg_{\pi_0}:\Hom(H_2(\Gamma,\mathbb{Z}),k) \to H^2(\Gamma,k).\]
specialized to the extension $\pi_0$. 

\begin{lem} \label{lem:split}
  For a Schur covering extension $\pi_0$ of $\Gamma$ the corresponding transgression $\tg_{\pi_0}$ splits \eqref{eq:uct} and provides an explicit decomposition 
\[ H^2(\Gamma,k) =\tg_{\pi_0}(\Hom(H_2(\Gamma,\bbZ),k))\oplus \Ext^1_\bbZ(\Gamma/[\Gamma,\Gamma],k). \]
\end{lem}

\begin{proof}
We need to show that $h\circ \tg_{\pi_0}=\id$. Since $h$ is surjective, we may fix an arbitrary extension $\pi$, as in \eqref{eq:pik}, representing an element of $H^2(\Gamma,k)$, and show that $h(\pi)=h\circ \tg_{\pi_0}\circ h (\pi)$. For convenience, we write $\pi'=\tg_{\pi_0}\circ h (\pi)$ and argue to show $h(\pi)=h(\pi')$.
Note that $h(\pi)$ is given in the diagram \eqref{eq:h}.
Specializing \eqref{eq:beta} to $\rho=\pi_0$ and $\beta=h(\pi)$, we obtain 
\[ 
 \begin{tikzcd}
      1 \arrow[r] & H_2(\Gamma,\mathbb{Z}) \arrow[r] \arrow[d, "h(\pi)"] & L_0 \arrow[d]\arrow[r, "\pi_0"] & \Gamma \arrow[d]\arrow[r] & 1 \\
            1 \arrow[r] & k \arrow[r] & L_0^{h(\pi)} \arrow[r, "\pi'"] & \Gamma \arrow[r] & 1 
\end{tikzcd}
\] 
which, when precomposed with \eqref{eq:E_0}, gives rise to 
\[ 
 \begin{tikzcd}
     1 \arrow[r] & H_2(\Gamma,\mathbb{Z}) \arrow[r] \arrow[d, "h(\pi)"] & \Gamma\wedge \Gamma \arrow[d, "\bar{\pi}'"]\arrow[r] & {[\Gamma,\Gamma]} \arrow[d]\arrow[r] & 1 \\
            1 \arrow[r] & k \arrow[r] & L_0^{h(\pi)} \arrow[r, "\pi'"] & \Gamma \arrow[r] & 1 
\end{tikzcd}
\] 
where we conclude the central vertical map to be $\bar{\pi}'$, as indicated, as it agrees on elements of the form $x\wedge y$ with the corresponding map defined in \eqref{eq:barpi}.
We deduce from \eqref{eq:h}, specialized to $\pi'$, that indeed $h(\pi)=h(\pi')$.
\end{proof}

Next, we consider the central extension $\bar{\id}$ given in~\eqref{eq:extsq} and the associated transgression $\tg_{\bar{\id}}\colon  \Hom(H_2(\Gamma,\bbZ),V) \to H^2([\Gamma,\Gamma],V)$.


\begin{lem} \label{lem:i}
  Given a group $\Gamma$ and an abelian group $k$, we have $\tg_{\bar{\id}}\circ h=i$, where $i:H^2(\Gamma,k)\to H^2([\Gamma,\Gamma],k)$ is
  the map associated with the inclusion $[\Gamma,\Gamma]\hookrightarrow\Gamma$.
\end{lem}

\begin{proof}
We fix a Schur covering $\pi_0$ of $\Gamma$ and obtain an associated splitting of $H^2(\Gamma,k)$ as in Lemma~\ref{lem:split}.
Since $\Ext^1_\bbZ(\Gamma/[\Gamma,\Gamma],k)$ is in the kernels of both $h$ and $i$,
we are left to show that $\tg_{\bar{\id}}\circ h=i$ on $\tg_{\pi_0}(\Hom(H_2(\Gamma,\bbZ),k))$.
By the naturality of the transgression, applied to the homomorphism of extensions \eqref{eq:E_0}, we obtain the following commutative square 
\[
\begin{tikzcd} 
      \Hom(H_2(\Gamma,\bbZ),k) \arrow[r, "\tg_{\pi_0}"] \arrow[d, "\id"] & \tg_{\pi_0}(\Hom(H_2(\Gamma,\bbZ),k)) \arrow[d, "i"]\arrow[ld, dashed, "h" '] \\
            \Hom(H_2(\Gamma,\bbZ),k) \arrow[r, "\tg_{\bar{\id}}"] & H^2([\Gamma,\Gamma],k)
\end{tikzcd}
\]
By Lemma~\ref{lem:split} we have the commutativity of the upper left triangle, which in turn gives the commutativity of the lower right triangle. This finishes the proof.
\end{proof}

\subsection{The properties $(T_2)$ and $[T_2]$}

\begin{lem}\label{lem: finite index inclusion}
  Let $\Gamma$ be a group and let $\Lambda<\Gamma$ be a finite index subgroup. 
  If $\Gamma$ satisfies property $(T_n)$ then $\Lambda$ also satisfies property $(T_n)$ and the restriction homomorphism $H^m(\Gamma,\bbC)\to H^m(\Lambda,\bbC)$ is an isomorphism for every $m\leq n$.
  If $\Gamma$ satisfies property $[T_n]$ then $\Lambda$ also satisfies property $[T_n]$.
\end{lem}

\begin{proof}
Let $m\leq n$.
Let $V$ be a unitary representation of $\Lambda$, and let $U$ be the corresponding induced unitary representation of $\Gamma$.
  By Shapiro's lemma, $H^m(\Lambda,V)\cong H^m(\Gamma,U)$.
  If $V$ has no non-trivial $\Lambda$-invariant vectors then $U$ has no non-trivial $\Gamma$-invariant vectors and by $(T_n)$, $H^m(\Gamma,U)=0$. We conclude that $H^m(\Lambda,V)=0$, thus $\Lambda$ has property $(T_n)$.
  In case $V=\bbC$, we have $U=L^2(\Gamma/\Lambda)=\bbC\oplus L^2_0(\Gamma/\Lambda)$, thus
  \[ H^m(\Lambda,\bbC)\cong H^m(\Gamma,\bbC)\oplus H^m(\Gamma, L^2_0(\Gamma/\Lambda))\]
  and the second summand vanishes by property $(T_n)$.
It follows that restriction homomorphism $H^m(\Gamma,\bbC)\to H^m(\Lambda,\bbC)$ is an isomorphism for every $m\leq n$.
  In particular, if $\Gamma$ satisfies property $[T_n]$ then $\Lambda$ also satisfies property $[T_n]$.
\end{proof}

We consider the central extension $\bar{\id}$ given in \eqref{eq:extsq} and the associated transgression $\tg_{\bar{\id}}: \Hom(H_2(\Gamma,\bbZ),V) \to H^2([\Gamma,\Gamma],V)$.

\begin{lem} \label{lem:tgT}
Let $\Gamma$ be a group which satisfies property $(T_2)$.
Then the associated transgression $\tg_{\bar{\id}}\colon \Hom(H_2(\Gamma,\bbZ),\bbC) \to H^2([\Gamma,\Gamma],\bbC)$
is an isomorphism.
\end{lem}

\begin{proof}
By the Universal Coefficient Theorem~\eqref{eq:uct} we see that $h$ is an isomorphism, as $\Gamma$ has finite abelianization.
  By Lemma~\ref{lem: finite index inclusion} we have that $i\colon H^2(\Gamma,\bbC)\to H^2([\Gamma,\Gamma],\bbC)$ is an isomorphism and by Lemma~\ref{lem:i} we conclude that $\tg_{\bar{\id}}$ is an isomorphism as well.
\end{proof}

\begin{thm} \label{thm: univext}
  Let $\Gamma$ be an $\mbox{FP}_2$ group which satisfies property $(T_2)$. Then $\Gamma\wedge \Gamma$ satisfies property $[T_2]$.
\end{thm}

\begin{proof}
Since $\Gamma$ has $(T_2)$, in particular, it has $(T)$. Hence $\Gamma$ has a finite abelianization. By Lemma~\ref{lem: finite index inclusion} we conclude that $[\Gamma,\Gamma]$ satisfies $(T_2)$. So $[\Gamma,\Gamma]$ has a finite abelianization, too.
    We set $A=H_2(\Gamma,\mathbb{Z})$.
     Since $\Gamma$ satisfies $\mbox{FP}_2$, it is finitely generated, thus $A$ is a finitely generated abelian group. 
     By~\cite[Proposition VIII.5.1]{Brown}, the finite index subgroup $[\Gamma,\Gamma]<\Gamma$ also satisfies $\mbox{FP}_2$. Using \cite[Proposition 2.2]{Kochloukova-Lima} we conclude from \eqref{eq:extsq} that $\Gamma\wedge \Gamma$ satisfies $\mbox{FP}_2$ as well. 
Let $V$ be a unitary $\Gamma\wedge\Gamma$-representation. 
By~\cite{Bader-Sauer}*{Theorem~3.7} the vanishing of $H^i(\Gamma\wedge\Gamma,W)$ for $i\in\{1,2\}$ and every unitary $\Gamma$-representation~$W$ is equivalent to the corresponding vanishing for every irreducible unitary $\Gamma$-representation~$W$. We may and will assume that $V$ is irreducible. 
In case $V=\bbC$ is the trivial representation, Lemma~\ref{lem:tgT} says that $\tg_{\bar{\id}}$
is an isomorphism. Thus $H^i(\Gamma\wedge\Gamma,V)=0$ by Lemma~\ref{lem:tg}, specialized to $L=\Gamma\wedge \Gamma$.
We may thus assume that $V$ is a non-trivial irreducible unitary $\Gamma$-representation. 

  We will show the vanishing of $H^i(\Gamma\wedge\Gamma,V)$, $i\in\{1,2\}$, via the Hochschild-Serre-Lyndon spectral sequence associated  with the central extension \eqref{eq:extsq}. Specifically, we show that 
  $E_2^{pq}=0$ for $p+q\le 2$ where 
   \[ E_2^{pq}=H^p([\Gamma,\Gamma], H^q(A,V)).\]
  By Schur's lemma, $A$ acts on $V$ via a character $\chi$. First suppose that $\chi$ is non-trivial. Then the group cohomology $H^\ast(A,V)=\operatorname{Ext}^\ast_{\bbC[A]}(\bbC, V)$ vanishes since multiplication with some $a\in A$ in any of the two variables of the Ext-term induces the same endomorphism of the group cohomology. Hence it is simultaneously the identity and multiplication with $\chi(a)$ which is only possible if the Ext-term vanishes. 
  We conclude that $E_2^{pq}=0$ for every $p,q$.
  Now suppose that $\chi$ is trivial. This means that $V$ is coming from a non-trivial  irreducible unitary $[\Gamma,\Gamma]$-representation.
  Since $[\Gamma, \Gamma]$ has property $(T_2)$, we have that $H^i([\Gamma,\Gamma],W)=0$ for every $i\in\{0,1,2\}$ and any unitary $[\Gamma,\Gamma]$-representation with no non-trivial invariant vectors. The $[\Gamma,\Gamma]$-modules $H^\ast(A, V)$ in the $E_2$-term of the spectral sequence are unitary $[\Gamma,\Gamma]$-representations with no invariant vectors since they are isomorphic to $H^\ast(A,\bbC)\otimes V$ (with the trivial action on $H^\ast(A,\bbC)$). So the $E_2^{pq}$-entries are zero for $p+q\le 2$. We conclude that, indeed, $H^i(\Gamma\wedge\Gamma,V)=0$ for $i\in\{1,2\}$.
  \end{proof}

\subsection{Central extensions of semisimple Lie groups and their lattices}



We will denote by $H^*_c$ the continuous group cohomology.
Recall that the fundamental group of a connected Lie group is always abelian.

\begin{prop} \label{prop:pi1}
  For a connected semisimple Lie group with finite center $G$,
  \[ \pi_1(G)\otimes \bbC\cong H^2_c(G,\bbC). \]
\end{prop}

In the proof below we will denote by $H^*_s$ the singular cohomology of topological spaces.

\begin{proof}
We let $K<G$ be a maximal compact subgroup, $U<G(\bbC)$ be a compact form of $G$ containing $K$, and $X=U/K$ be the corresponding dual symmetric space. By \cite[Theorem 1.1]{Mimura}, $G$ is homotopy equivalent to $K$. Thus by Hurewicz theorem, $\pi_1(G)\otimes \bbC\cong \pi_1(K)\otimes \bbC \cong H^1_s(K,\bbC)$. By the van Est isomorphism~\cite{Borel-Wallach}*{Theorem~IX.5.6},  $H^*_c(G,\bbC)\cong H^*_s(X,\bbC)$. The Serre spectral sequence associated with the fibration $U\to X$ and with $E_2$-term $E_2^{pq}=H^p_s(X,H^q_s(K,\bbC))$, converges to $H^*_s(U,\bbC)$.
As $U$ is a connected compact Lie group, the cohomology ring of $U$ is an exterior algebra generated by classes in odd degrees~\cite[Proposition 7.3]{Borel53}. Since $U$ is semisimple, we obtain that $H^1_s(U,\bbC)=H^2_s(U,\bbC)=0$. 
From the five terms exact sequence 
\[ 0 \to E_2^{10} \to H^1_s(U,\bbC) \to E_2^{01} \to E_2^{20} \to H^2_s(U,\bbC) \]
we deduce that $E_2^{01} \cong E_2^{20}$. Hence
\[ H^1_s(K,\bbC) \cong H^0_s(X,H^1_s(K,\bbC))  \cong H^2_s(X,H^0_s(K,\bbC))\cong H^2_s(X,\bbC). \]
Thus $\pi_1(G)\otimes \bbC\cong H^2_c(G,\bbC)$.
\end{proof}

\begin{prop} \label{prop:tilde}
  Let $G$ be a connected semisimple Lie group with finite center and no compact factors. Let $\Gamma<G$ be a lattice.
  Denote by $\tilde{G}\to G$ the universal cover and let $\pi:\tilde{\Gamma}\to \Gamma$ be the pull back of $\Gamma$.
  If $G$ has property $(T)$ and real rank at least 3
  then the natural map $\bar{\pi}:\Gamma \wedge \Gamma \to [\tilde{\Gamma},\tilde{\Gamma}] < \tilde{\Gamma}$ has a finite kernel and image of finite index.
\end{prop}

\begin{rem}
We note that for $G=\text{SL}_3(\bbR)$, $\tilde{G}\to G$ is a double cover,
hence so is $\tilde{\Gamma}\to \Gamma$ for any lattice $\Gamma<G$.
However, by \eqref{eq:extsq}, the kernel of $\Gamma\wedge \Gamma \to [\Gamma,\Gamma]$ is isomorphic to $H_2(\Gamma,\bbZ)$, which varies dramatically.
This follows from a computation by Soul\'{e} of $H^2(\Gamma,\bbC)$.
Indeed, while $H^2(\text{SL}_3(\bbZ),\bbC)=0$ by \cite[Theorem~4(iii)]{Soule}, there is a sequence of congruence subgroups $\Gamma_i<\text{SL}_3(\bbZ)$ such that the second Betti number of $\Gamma_i$ tends to~$\infty$ by \cite[Theorem~7(ii)]{Soule}. See also Remark~\ref{rem: infinite virtual betti}. 
\end{rem}

\begin{proof}
  We note that $\tilde{\Gamma}$ has property $(T)$ by \cite[Theorem 3.5.2]{Bekka}.
  We identify the kernel of $\pi:\tilde{\Gamma}\to \Gamma$ with $\pi_1(G)$.
  By property $(T)$, both $[\tilde{\Gamma},\tilde{\Gamma}]<\tilde{\Gamma}$ and $[\Gamma,\Gamma]<\Gamma$ are of finite index, thus
  we identify the kernel of the restriction map $\pi:[\tilde{\Gamma},\tilde{\Gamma}] \to [\Gamma,\Gamma]$
  with a finite index subgroup of $\pi_1(G)$.

  By the universal property of $\Gamma\wedge \Gamma$ we have $\bar{\id}=\pi \circ \bar{\pi}:\Gamma\wedge \Gamma\to [\Gamma,\Gamma]$, which is surjective.
  Since the kernel of $\pi$ is central,
  the image of $\bar{\pi}$
  is normal and $[\tilde{\Gamma},\tilde{\Gamma}]/\bar{\pi}(\Gamma\wedge \Gamma)$ is abelian.
  Using $(T)$, we conclude that the image of $\bar{\pi}$ is of finite index.

  The kernel of $\bar{\id}$ is $H_2(\Gamma;\bbZ)$ and the kernel of $\pi$ is a finite index subgroup of $\pi_1(G)$.
  These two finitely generated abelian groups have the same rank by Proposition~\ref{prop:pi1} since the restriction map
  $H^2(G,\bbC)\to H^2(\Gamma,\bbC)$ is an isomorphism by \cite{Bader-Sauer}*{Theorem~F} and $\rank_\bbR G\ge 3$.
  We conclude that $\bar{\pi}$ has a finite kernel.
\end{proof}

\begin{thm} \label{thm:tildeT2}
  Let $G$ be a semisimple Lie group with finite center and no compact factors. Let $\Gamma<G$ be a lattice.
  Denote by $\tilde{G}\to G$ the universal cover and let $\tilde{\Gamma}\to \Gamma$ be the pull back of $\Gamma$.
  If all simple factors of $G$ are of real rank at least 3
  then $G$ and $\Gamma$ have property $(T_2)$ and $\tilde{G}$ and $\tilde{\Gamma}$ have property $[T_2]$.
\end{thm}

\begin{proof}
  The fact that $G$ and $\Gamma$ have property $(T_2)$ is proved in \cite{Bader-Sauer}.
  By Theorem~\ref{thm: univext}, we get that $\Gamma\wedge \Gamma$ has property $[T_2]$
  and by Proposition~\ref{prop:tilde} we deduce that also $\tilde{\Gamma}$ has property $[T_2]$.
  Specializing to a cocompact lattice $\Gamma<G$, it also follows, by Shapiro Lemma, that  $\tilde{G}$ has the property $[T_2]$, as then $\tilde{\Gamma}<\tilde{G}$ is a cocompact lattice satisfying property $[T_2]$.
\end{proof}

\subsection{A recollection of CSP and a result of Deligne}

A lattice $\Gamma<G$ being \emph{arithmetic} means that there is a number field $k$, a semisimple simply connected
$k$-algebraic subgroup $\bbH$ of $\GL_n$ such that there is an epimorphism $\prod_{s\in V_\infty(k)}\bbH(k_v)\twoheadrightarrow G$ with compact kernel and the image of $\bbH(\calO)$ is commensurable with $\Gamma$. Here $V_\infty(k)$ is the set of Archimedean valuations of $k$ and $\calO$ is the ring of integers of $k$.

The lattice $\Gamma<G$ is said to have the \emph{congruence subgroup property (CSP)} if the kernel of $\widehat{\bbH(\calO)}\to \bbH(\hat\calO)$ is finite where the symbol indicates the profinite completion, respectively. The intersection of all finite index subgroups of $\Gamma$ is called the \emph{profinite radical} of $\Gamma$, denoted $\prof(\Gamma)$.
The following theorem was essentially proved by Deligne~\cite{De}. For a clear and concise exposition see~\cite{Witte} as well as~\cite[\S5.2]{frob}.

\begin{thm}[Deligne]\label{thm: deligne}
  Let $\Gamma<G$ be a lattice satisfying the congruence subgroup property. Assume that $\pi_1(G)$ is infinite. Let  $\tilde G\to G$ be the universal cover, which is the universal central extension of $G$. Then the preimage $\tilde\Gamma$ is not residually finite. Furthermore,
  \[\prof(\tilde\Gamma)\subset\ker\bigl(\tilde\Gamma\to \Gamma\bigr)\]
  is a subgroup of finite index.
\end{thm}

\subsection{{Proof of Theorem~\ref{thm: main stability result}}}

A finitely generated linear group is residually finite by Malcev's theorem.
Accordingly, for a finitely generated group $\Gamma$, every homomorphism to a linear group factors via $\Gamma/\prof(\Gamma)$.  

\begin{lem} \label{lem:profker}
  Let $\Gamma$ be a finitely generated group. Let
  $N\lhd \Gamma$ be normal subgroup of $\Gamma$ contained in $\prof(\Gamma)$.
  If $\Gamma$ is Frobenius stable then so is $\Gamma/N$.
\end{lem}

\begin{proof}
  Let $\phi=(\phi_n)$ be an asymptotic homomorphism of $\Gamma/N$. Let $p\colon\Gamma\to \Gamma/N$ be the projection. So $\phi\circ p=(\phi_n\circ p)$ is an asymptotic homomorphism of $\Gamma$. By assumption, $\phi\circ p$ is approximated by a sequence of true homomorphisms $\psi=(\psi_n)$. For every $n\in\bbN$, the image $\psi_n(\Gamma)$ is a finitely generated linear group, hence residually finite. Thus each $\psi_n$ factors through $\Gamma/\prof(\Gamma)$, in particular, through $\Gamma/N$. The induced homomorphisms on $\Gamma/N$ approximates $\phi$.
\end{proof}

\begin{lem} \label{lem:finind}
  Let $\Gamma$ be a finitely presented group and let $N\lhd \Gamma$ be a finite normal subgroup.
  If $\Gamma$ is Frobenius stable then also $\Gamma/N$ is Frobenius stable.
\end{lem}

In the proof below, $\|\cdot\|$ means the Frobenius norm.

\begin{proof}
  We let $F_S\to \Gamma$ be a surjection from the free group on a finite set $S$
  and assume as we may that there exists a subset $\bar{N}\subseteq S$ which image is $N$.
  Let $R\subset F_S$ be a finite set of relations for $\Gamma$ and let $T=R\cup \bar{N} \subset F_S$, which is a finite set of relations for $\Gamma/N$.

  For every $\epsilon>0$ there exists $\delta>0$ such that every Frobenius $\delta$-almost representation of $\Gamma$ is Frobenius $\epsilon$-close to a representation.
  That is, for every homomorphism $\rho:F_S\to U(n)$ such that for every $r\in R$, $\|\rho(r)-1\|<\delta$, there exists a homomorphism $\rho':F_S\to U(n)$ such that for every $r\in R$, $\rho'(r)=0$ and for every $s\in S$, $\|\rho(s)-\rho'(s)\|<\epsilon$.

  We let $\alpha$ be the minimum of $\max_{n\in N} \|\pi(n)-1\|$, ranging over all non-trivial irreducible representations $\pi$ of $N$. 

  We fix $\epsilon>0$. Without loss of the generality, we assume $\epsilon\leq \alpha/2$.
  We choose $\delta$ accordingly and again, without loss of the generality, we assume $\delta\leq \alpha/2$.
  Let $\rho:F_S\to U(n)$ be a homomorphism such that for every $t\in T$, $\|\rho(t)-1\|<\delta$.
  In particular, for every $r\in R$, $\|\rho(r)-1\|<\delta$, thus there exists a homomorphism $\rho':F_S\to U(n)$ such that for every $r\in R$, $\rho'(r)=1$ and for every $s\in S$, $\|\rho(s)-\rho'(s)\|<\epsilon$.
  We view $\rho'$ as a representation of $\Gamma$.
  Restricting to $N$, we get that for every $n\in N$,
  \[ \|\rho'(n)-1\|\leq \|\rho(n)-\rho'(n)\|+\|\rho(n)-1\|< \epsilon+\delta\leq \alpha. \]
  We note that the Frobenius norm of a representation majorizes the Frobenius norm of any subrepresentation.
  It follows that $\rho'|_N$ contains no non-trivial irreducible subrepresentation, thus for every $t\in T$, $\rho'(t)=1$.
  We conclude that $\rho'$ descends to a representation of $\Gamma/N$, which is $\epsilon$ close to $\rho$.
\end{proof}

\begin{cor} \label{cor:N}
  Let $\Gamma$ be a finitely generated group and let $N\lhd \Gamma$ be a normal subgroup such that $\Gamma/N$ is finitely presented.
  Assume that the intersection of $N$ with the profinite radical of $\Gamma$ is of finite index in $N$.
  If $\Gamma$ is Frobenius stable then also $\Gamma/N$ is Frobenius stable.
\end{cor}

\begin{proof}
  Let $N_0$ be the intersection of $N$ with the profinite radical of $\Gamma$.
  Then $\Gamma/N_0$ is an extension of $\Gamma/N$ by the finite group $N/N_0$, hence it is finitely presented. By Lemma~\ref{lem:profker}, applied for $N_0\lhd \Gamma$, $\Gamma/N_0$ is Frobenius stable and by Lemma~\ref{lem:finind}, applied to $N/N_0\lhd \Gamma/N_0$, $\Gamma/N$ is Frobenius stable.
\end{proof}

\begin{proof}[Proof of Theorem~\ref{thm: main stability result}]
  Let $\Gamma<G$ be as in Theorem~\ref{thm: main stability result}.
  Then $\Gamma$ has $(T_2)$ by \cite{Bader-Sauer}.
  If no factor of $G$ is of Hermitian type then $H^2(G,\bbC)=0$ and by \cite{Bader-Sauer} also $H^2(\Gamma,\bbC)=0$, thus $\Gamma$ has $[T_2]$, so it is Frobenius stable by Theorem~\ref{thm: main old result}.
  From now on we assume that at least one of the factors of $G$ is of Hermitian type, thus $\pi:\tilde{G}\to G$ has an infinite kernel. We denote $\tilde{\Gamma}=\pi^{-1}(\Gamma)$.
  By Theorem~\ref{thm:tildeT2}, $\tilde{\Gamma}$ is a $[T_2]$-group and by Theorem~\ref{thm: main old result} it is Frobenius stable.
  We let $N\lhd \tilde{\Gamma}$ be the kernel of the extension $\tilde{\Gamma} \to \Gamma$.
  By Corollary~\ref{cor:N}, if the intersection of $N$ with the profinite kernel of $\tilde{\Gamma}$ is of finite index in $N$ then $\Gamma$ is Frobenius stable.
  By Theorem~\ref{thm: deligne}, this is indeed the case, as $\Gamma$ is assumed to have the congruence subgroup property. Thus $\Gamma$ is Frobenius stable.
\end{proof}

\section{Operator instability}

The following is essentially due to J. Wolf and A. Borel.

\begin{thm} \label{thm:css}
Let $U$ be a compact connected Lie group, $K<U$ a closed subgroup containing no normal subgroup of $U$ and assume that $M=U/K$ is a Riemannian symmetric space.
Then $M$ is an odd dimensional rational homology sphere if and only if, up to a cover, either $U=\SO(n+1)$ and $K=\SO(n)$ for some odd $n$ or $U=\SU(3)$ and $K=\SO(3)$.  
\end{thm}

\begin{proof}
    It follows from \cite[Theorem 1]{Wolf} that, up to a cover, either $M$ is an actual odd dimensional sphere or $M=\SU(3)/\SO(3)$. 
    However, \cite[Theorem 3]{Borel} classifies the pairs $K<U$ for which $U/K$ is an odd dimensional sphere, and while comparing to the classification of compact symmetric spaces, one gets that the only fitting pair is $U=\SO(n+1)$ and $K=\SO(n)$.    
\end{proof}

\begin{rem} \label{rem:homsphere}
    Note that a compact manifold $M$ is an odd dimensional rational homology sphere if and only if for all even $i>0$, $H^i_s(M,\mathbb{Q})= 0$, where $H^*_s$ denotes the singular cohomology of $M$.
The ``only if" part is obvious and the ``if" part follows by Poincare duality.
Indeed, assuming that for all even $i>0$, $H^i_s(M,\mathbb{Q})= 0$ and $M$ is not an odd dimensional rational homology sphere, if $n=\dim(M)$ is even we get a contradiction by setting $i=n$, and if $n$ is odd we get some $j<n$ with $H^j(\Gamma,\mathbb{Q})\neq 0$, thus also $H^{n-j}(\Gamma,\mathbb{Q})\neq 0$ by Poincare duality, and either $j$ is even of $n-j$ is even.
\end{rem}

For a locally compact group $G$, we denote by $H^*_c$ the continuous group cohomology.

\begin{cor} \label{cor:sshs}
   Let $G$ be a semisimple real Lie group not locally isomorphic to $\SO(n,1)$ for some odd $n$ or $\SL_3(\mathbb{R})$. Then for some even $i>0$, $H^i_c(G,\mathbb{R})\neq 0$.
\end{cor}

\begin{proof}
    Follows at once from Theorem~\ref{thm:css}, by Remark~\ref{rem:homsphere} and the fact that the continuous group cohomology of $G$ is isomorphic to the singular cohomology of its dual symmetric space.
\end{proof}

\begin{thm} \label{thm:ohs}
    Let $G$ be a semisimple real Lie group not locally isomorphic to $\SO(n,1)$ for some odd $n$ or $\SL_3(\mathbb{R})$, and let $\Gamma<G$ be a cocompact lattice. 
    Then for some even $i>0$ we have $H^i(\Gamma,\mathbb{R})\neq 0$.
\end{thm}

\begin{proof}
    By Matsushima's formula~\cite{Matsushima} we have an injection $H^*_c(G,\mathbb{R})\hookrightarrow H^*(\Gamma,\mathbb{R})$,
    thus this follows by Corollary~\ref{cor:sshs}.
\end{proof}

\begin{proof}[Proof of Theorem~\ref{thm:mainO}]
Immediate by Theorem~\ref{thm:ohs} and Theorem~\ref{thm: main old result}(O).
\end{proof}

\subsection{Lattices in $\SL_3(\mathbb{R})$ and $\SO(n,1)$}

\begin{thm} \label{thm:SL3}
    Let $\Gamma$ be a lattice in $\SL_3(\mathbb{R})$. 
    Then there exists a finite index subgroup $\Gamma_1 \leq \Gamma$ satisfying $H^2(\Gamma_1,\mathbb{Q})\neq 0$.
    Moreover, for every finite index subgroup $\Gamma_2 \leq \Gamma_1$, $H^2(\Gamma_2,\mathbb{Q})\neq 0$.
\end{thm}

\begin{proof}
We first note that if $\Gamma_2 \leq \Gamma_1$ is of finite index then $H^*(\Gamma_1,\mathbb{Q})$ injects into $H^*(\Gamma_2,\mathbb{Q})$ by Shapiro's Lemma.
This proves the ``moreover" part.  

The Euler characteristic of $G$ and its lattices is 0 and the homological dimension is bounded by 5, the dimension of the corresponding symmetric space.
We denote $b_n=H^n(\Gamma,\mathbb{Q})$ and note that $b_1=0$ by property T.

The lattice $\Gamma=\SL_3(\mathbb{Z})$ is of $\mathbb{Q}$-rank 2, hence of rational cohomological dimension 3, that is $b_4=b_5=0$, and the same applies also to its subgroups of finite index. 
It follows that for these groups $b_2=b_3-1$.
By \cite[Theorem 2]{Ash}, the kernel of $\SL_3(\mathbb{Z}) \to \SL_3(\mathbb{Z}/7)$ has $b_3>5815$, hence $b_2>5814$ (see also \cite[Theorem 7]{Soule}.
By Shapiro's Lemma the theorem follows for every lattice which is commensurable with $\SL_3(\mathbb{Z})$, thus for every non-uniform lattice.

We thus may and will assume that $\Gamma$ is cocompact.
Without loss of the generality, let $\Gamma$ be torsion-free.
By Poincare duality we have $b_5=b_0=1$, $b_4=b_1=0$ and $b_3=b_2$.
By \cite[Theorem B]{Lub}, there are finite quotients of $\Gamma$ with arbitarily large $2$-Sylow subgroups (actually for every prime). In particular, 
there exists a finite index subgroup $\Gamma_0\subseteq \Gamma$ which surjects on $\mathbb{Z}/2\times \mathbb{Z}/2$. We let $\Gamma_1\lhd \Gamma_0$ be the kernel of this surjection and consider $M=\Gamma_1\backslash G/K$. Note that $\Gamma_0/\Gamma_1$ acts on $M$. By \cite[Theorem D]{Davis} (due to Davis and Weinberger) we get that $M$ is not a rational homology sphere and we conclude that $b_2(\Gamma_1)\neq 0$.
\end{proof}
\begin{rem}\label{rem: infinite virtual betti}
Since the compact dual of $\SL_3(\bbR)$ is a rational homology sphere, the preceding theorem and a theorem of Venkataramana~\cite{venka} imply that the virtual second Betti number of a lattice is $\SL_3(\bbR)$ is infinite. 
\end{rem}
\begin{cor} \label{cor:SL3}
    Let $\Gamma$ be a lattice in $\SL_3(\mathbb{R})$. 
    Then there exists a finite index subgroup $\Gamma_1 \leq \Gamma$
    such that every finite index subgroup $\Gamma_2 \leq \Gamma_1$ is not operator stable.
\end{cor}

\begin{proof}
    Immediate by Theorem~\ref{thm:SL3} and Theorem~\ref{thm: main old result}(O).
\end{proof}

The use of \cite[Theorem D]{Davis} in the proof of Theorem~\ref{thm:SL3} used the fact that the dimension of the symmetric space of $\SL_3(\mathbb{R})$ is 5, which is 1 mod 4.
In a similar manner we get the following.

\begin{prop} \label{prop:SOn1}
        Let $\Gamma$ be a cocompact lattice in $\SO(n,1)$ for $n=4k+1$.
    Then there exists an even $i>0$ and a finite index subgroup $\Gamma_1 \leq \Gamma$ satisfying $H^i(\Gamma_1,\mathbb{Q})\neq 0$.
    Moreover, for every finite index $\Gamma_2 \leq \Gamma_1$, $H^i(\Gamma_2,\mathbb{Q})\neq 0$.
\end{prop}

\begin{proof}
The dimension of the symmetric space of $\SO(n,1)$ is 1 mod 4.
Applying \cite[Theorem B]{Lub} and \cite[Theorem D]{Davis}, just as in the proof of Theorem~\ref{thm:SL3},
we get that $M=\Gamma_1\backslash G/K$ is not a rational homology sphere, and the proof follows by Remark~\ref{rem:homsphere}.
\end{proof}

The following follows from the celebrated solution by Agol and Wise of the virtual Haken conjecture. 

\begin{thm}[{\cite{Agol}}] \label{thm:SO31}
    Let $\Gamma$ be a cocompact lattice in $\SO(3,1)$.
    Then there exists a finite index subgroup $\Gamma_1 \leq \Gamma$ such that for every $0 \leq i \leq 3$, $H^i(\Gamma_1,\mathbb{Q})\neq 0$.
    Moreover, for every finite index $\Gamma_2 \leq \Gamma_1$, $H^i(\Gamma_2,\mathbb{Q})\neq 0$.
\end{thm}

The following is due to Bergeron and Clozel.

\begin{thm}[{\cite[Corollary 1.8]{Ber-Clo13}}] \label{thm:BC}
    Let $\Gamma$ be an arithmetic lattice in $\SO(n,1)$, $n\neq 3$.
    In case $n= 7$, assume $\Gamma$ is not of the form ${}^{6}D_{4}$.
    Then there exists a finite index subgroup $\Gamma_1 \leq \Gamma$ such that for every $0 \leq i \leq n$, $H^i(\Gamma_1,\mathbb{Q})\neq 0$.
    Moreover, for every finite index $\Gamma_2 \leq \Gamma_1$, $H^i(\Gamma_2,\mathbb{Q})\neq 0$.
\end{thm}

\begin{rem}
 In \cite[Theorem 1.1]{Ber-Clo17} Bergeron and Clozel show that for lattices in $\SO(7,1)$ of the form ${}^{6}D_{4}$, it holds that all congruence subgroups have $b_1=0$. The corresponding questions regarding non-congruence subgroups and higher cohomologies are still open. 
\end{rem}

\begin{thm}[{\cite[Theorem 3.5]{Lub96}}] \label{thm:nonar}
   For $n>3$ and a lattice $\Gamma< G=\SO(n,1)$, if $M=\Gamma\backslash G/K$ admits a codimension 1 totally geodesic submanifold then there exists a finite index subgroup $\Gamma_1 \leq \Gamma$ such that $H^1(\Gamma_1,\mathbb{Q})\neq 0$.
   Moreover, for every finite index $\Gamma_2 \leq \Gamma_1$, $H^1(\Gamma_2,\mathbb{Q})\neq 0$.
\end{thm}

We note that for $n>3$, for all \emph{known} non-arithmetic lattices $\Gamma< G=\SO(n,1)$, $M=\Gamma\backslash G/K$ admits a codimension 1 totally geodesic submanifold.

\begin{cor} \label{cor:SOn1}
    Let $\Gamma$ be a cocompact lattice in $G=\SO(n,1)$, $n$ odd.
    Assume either one of the following conditions:
    \begin{itemize}
        \item $n=3$.
        \item $n=4k+1$.
        \item $n=4k+3$ and $\Gamma$ is arithmetic, but not of the form ${}^{6}D_{4}$ in case $n=7$.
        \item $\Gamma\backslash G/K$ has a codimension 1 totally geodesic submanifold. 
    \end{itemize}
    Then there exists a finite index subgroup $\Gamma_1 \leq \Gamma$
    such that every finite index subgroup $\Gamma_2 \leq \Gamma_1$ is not operator stable.
\end{cor}

\begin{proof}
    Immediate by Theorem~\ref{thm: main old result}(O), given Theorem~\ref{thm:SO31}, Theorem~\ref{thm:BC} and Theorem~\ref{thm:nonar} (taking into account Remark~\ref{rem:homsphere}).
\end{proof}

\section{Remarks and suggestions for future research}

\subsection{Non-uniform lattices}

Theorem~\ref{thm:mainO} gives a pretty good picture of the operator instability of uniform lattices.
The situation with non-uniform lattices is a priori different, as the next proposition applies to all non-uniform lattices in $\SL_2(\mathbb{R})$.

\begin{prop}
    Every virtually free group is operator stable.
\end{prop}

\begin{proof}
    This follows from \cite[Remark 5.2]{frob}, as virtually free groups are of rational cohomological dimension 1.    
\end{proof}

However, we believe that non-uniform lattices in simple Lie groups other then $\SL_2(\mathbb{R})$ are not operator stable.
In fact, we can show it in many cases by some ad hoc arguments, but we lack a unified treatment. Let us illustrate this with an easy example. 

\begin{prop}
    For $n\geq 16$, every lattice in $\SL_n(\mathbb{R})$ is not operator stable.
\end{prop}

\begin{proof}
    Setting $G=\SL_n(\mathbb{R})$ and letting $\Gamma<G$ be a lattice, we have by \cite{Bader-Sauer} that the restriction homomorphism $H^i_c(G,\mathbb{C})\to H^i(\Gamma,\mathbb{C})$ is an isomorphism for $i<\rank(G)=n-1$, and in particular for $i=14$.
    It is well known that $H^i_c(G,\mathbb{C})\cong H^i_s(\SU(n)/\SO(n),\mathbb{C})$ and the cohomology ring of $\SU(n)/\SO(n)$ is an exterior algebra.
    By \cite[Theorem III.6.7 on p. 149]{mimura+toda} it has generators in degrees 5 and 9, hence also an element in degree 14=5+9. Thus $H^{14}(\Gamma,\mathbb{C}) \neq 0$ and we conclude by Theorem~\ref{thm: main old result}(O) that $\Gamma$ is not operator stable.
\end{proof}

A more sophisticated method by Ash-Borel~\cite[Theorem 3.1]{ash+borel} in combination with Theorem~\ref{thm: main old result}(O) yields the following result (and actually much more) by producing even dimensional virtual cohomology classes via modular symbols. 

\begin{prop}
Let $\calO$ be the ring of integers in a number field. Then $SL_n(\calO)$ for $n> 2$ contains a finite index subgroup that is not operator stable. 
\end{prop}
\subsection{$p$-Schatten norms}

The $p$-Schatten norm of $M_n(\mathbb{C})$ is $\|A\|_p=(\mbox{tr}((A^*A)^p))^{1/p}$.
This generalizes the Frobenius norm, obtained for $p=2$.
We define the metric $d^p_n$ on $U(n)$ by $d^p_n(g,h)=\|g-h\|_p$
and set $\frak{g}_p=(U(n),d^p_n)$.

In \cite{frob} the lattices in $\ell$-adic simple Lie groups of rank at least 3 where shown to be Frobenius stable. In \cite{LO} this result was extended to any $p$-Schatten norm, provided $\ell$ is large enough with respect to $p$.
The starting point of this paper was to extend \cite{frob} from $\ell$-adic Lie groups to real Lie groups by replacing Garland's theory by \cite{Bader-Sauer} (and references therein).
This applies only to unitary representation and we do not know how to extend it to other Banach spaces, i.e, we do not know how to prove the $p$-Schatten stability for any $p\neq 2$ and any higher rank lattice in a real simple Lie group.

\subsection{$S$-arithmetic groups and $l$-adic groups}

A uniform lattice $\Gamma$ in a (virtual) product of locally compact groups $G_1$ and $G_2$ for which $H^{\ast>0}_c(G_i,\bbC)\ne 0$ 
is operator instable: By the K\"unneth formula the continuous cohomology of $G$ is non-zero in an even positive degree. By the Shapiro lemma it injects into the cohomology of $\Gamma$, and so we can apply Theorem~\ref{thm: main old result}(O). 

We apply this observation to the case of $S$-arithmetic groups. 
In the following, let $k$ be a number field, $\mathcal{O}<k$ the ring of integers, $S$ a finite set of primes containing $S_\infty$, the Archimedean primes and let ${\bf G}$ be a $k$-algebraic group. 
We let $\Gamma={\bf G}(\mathcal{O}_S)={\bf G}(k)\cap \GL(\mathcal{O}_S)$ for some embedding $\bf G\hookrightarrow \GL_n$. Note that $\Gamma$ is a lattice in $G:=\prod_{\nu\in S}{\bf G}(k_\nu)$. An immediate consequence of the above observation is:

\begin{prop}
    If ${\bf G}(k_{\nu_i})$ is non-compact for two different place $\nu_1,\nu_2\in S$ and $\Gamma<G$ is uniform, then 
    $\Gamma$ is operator instable.
\end{prop}

By a remarkable result of Blasius-Franke-Grunewald~\cite[Theorem 1]{Blasius} (note that the tensor symbol $\otimes$ should be a sum symbol $\oplus$ in \emph{loc.~cit.}) the assumption that $\Gamma$ is uniform could be removed in many cases: 

\begin{prop}
    Asssume that $\bf G$ is simply connected (as a $k$-group). 
    Assume that $\mathbf{G}(k_\infty)$ is not compact and and there exists a non-Archimedean place $\nu\in S$ such that $\bfG(k_\nu)$ is non-compact. 
    Assume further that ${\bf G}(k_\infty)$ is neither locally isomorphic to $\SL_3(\mathbb{R})$ nor to $\SO(n,1)$ for some odd $n$.
    Then $\Gamma$ is operator instable.
\end{prop}

\begin{proof}
According to~\cite{Blasius}*{Theorem~1} the cohomology of $\mathbf{G}(k_\infty)$ injects into the cohomology of $\Gamma$. 
We conclude by Corollary~\ref{cor:sshs} and Theorem~\ref{thm: main old result}(O). 
\end{proof}

The current paper puts the real case more or less in line with the $\ell$-adic case as far as the Frobenius norm is concerned.
For the operator norm it even puts the real case ahead of the $\ell$-adic case. While our Theorem~\ref{thm:mainO} shows that almost all uniform lattices in real semisimple Lie groups are operator instable, our understanding of the $\ell$-adic case is far from being complete.
If $G$ is a simple $\ell$-adic group of rank $r$ and $\Gamma<G$ a lattice (necessarily cocompact),
then by Garland's theorem, $H^i(\Gamma,\mathbb{C})=0$ for every $0<i<r$, while $H^r(\Gamma,\mathbb{C})\neq 0$.
Thus if $r$ is even, $\Gamma$ is operator instable by Theorem~\ref{thm: main old result}(O).
However, for odd $r$ we know nothing!

\end{document}